\documentclass[11pt]{article}

\usepackage{amssymb,amsthm,amsmath}

\usepackage{nicefrac}

\usepackage[usenames,dvipsnames]{xcolor}
\usepackage[colorlinks,citecolor=blue,linkcolor=BrickRed]{hyperref}\usepackage[colorlinks,citecolor=blue,linkcolor=BrickRed]{hyperref}
\usepackage{thmtools,thm-restate}
\usepackage{comment}
\usepackage{cleveref}
\usepackage{fullpage}
\usepackage{graphicx}
\usepackage{graphicx}
\usepackage{subcaption}

\usepackage{algorithm}
\usepackage{algpseudocode}

\newtheorem{theorem}{Theorem}[section]
\newtheorem{corollary}[theorem]{Corollary}

\newtheorem{lemma}[theorem]{Lemma}
\newtheorem{observation}[theorem]{Observation}
\newtheorem{definition}[theorem]{Definition}

\newtheorem{remark}[theorem]{Remark}

\newtheorem*{claim*}{Claim}

\newcommand{\E}{\mathbb{E}}

\newcommand{\R}{\mathbb{R}}
\newcommand{\C}{\mathbb{C}}
\newcommand{\Z}{\mathbb{Z}}

\newcommand{\ind}{\mathbf{1}}

\newcommand{\vol}{\mathsf{vol}}

\newcommand{\disc}{\mathsf{disc}}
\newcommand{\wdisc}{\mathsf{wdisc}}

\newcommand{\ignore}[1]{{}}

\newcommand{\err}{\mathsf{err}}

\newcommand{\SO}{\mathsf{SO}}
\newcommand{\WSO}{\mathsf{WSO}}
\newcommand{\HK}{\mathsf{HK}}

\newcommand{\F}{\mathcal{F}}

\newcommand{\D}{\mathcal{D}}

\newcommand{\nnz}{\mathsf{nnz}}

\newcommand{\SubgVecBal}{\mathsf{BalSubgDisc}}

\newcommand{\wSubgTrans}{\mathsf{WSubgTrans}}
\newcommand{\SubgTrans}{\mathsf{SubgTrans}}

\newcommand{\bfd}{\mathbf{d}}
\newcommand{\bfwd}{\mathbf{wd}}

\newcommand{\bfs}{\mathbf{s}}
\newcommand{\bfj}{\mathbf{j}}

\newcommand{\bx}{\boldsymbol{x}}
\newcommand{\bv}{\boldsymbol{v}}
\newcommand{\bz}{\boldsymbol{z}}

\newcommand{\bk}{\boldsymbol{k}}
\newcommand{\ba}{\boldsymbol{a}}
\newcommand{\bw}{\boldsymbol{w}}
\newcommand{\be}{\boldsymbol{e}}
\newcommand{\bs}{\boldsymbol{s}}
\newcommand{\bc}{\boldsymbol{c}}
\newcommand{\by}{\boldsymbol{y}}

\allowdisplaybreaks

{\hspace*{\fill}$\Box$\par}

\title{High-Dimensional Quasi-Monte Carlo\\ via Combinatorial Discrepancy}

\author{Jiaheng Chen\thanks{University of Chicago, Chicago, IL, USA. \texttt{jiaheng@uchicago.edu}.} 
\and
Haotian Jiang\thanks{University of Chicago, Chicago, IL, USA. \texttt{jhtdavid@uchicago.edu}.}
\and
Nathan Kirk\thanks{Illinois Institute of Technology, Chicago, IL, USA. \texttt{nkirk@illinoistech.edu}.}
}

\date{}

\begin{document}

\maketitle
 \begin{abstract}
 \thispagestyle{empty}
Monte Carlo (MC) and Quasi-Monte Carlo (QMC) methods are classical approaches for the numerical integration of functions $f$ over $[0,1]^d$. While QMC methods can achieve faster convergence rates than MC in moderate dimensions, their tractability in high dimensions typically relies on additional structure---such as low effective dimension or carefully chosen coordinate weights---since worst-case error bounds grow prohibitively large as $d$ increases. 

In this work, we study the construction of high-dimensional QMC point sets via combinatorial discrepancy, extending the recent QMC method of Bansal and Jiang \cite{BJ25}. We establish error bounds for these constructions in weighted function spaces, and for functions with low effective dimension in both the superposition and truncation sense. We also present numerical experiments to empirically assess the performance of these constructions.
\end{abstract}

{\hypersetup{linkcolor=BrickRed}
}
 \thispagestyle{empty}

\maketitle

\setcounter{page}{1}
 \allowdisplaybreaks

\section{Introduction}

High-dimensional integration appears in statistics \cite{Gen93, Gen92, FanWan94}, quantum physics \cite{quantumMC2011}, financial mathematics \cite{wangsloan05, LEc09, paskov97}, uncertainty quantification in engineering \cite{Kei96, MorCaf95} and computer graphics and vision \cite{Vea97, paulin2022, kellprem12}, among many others. After suitable transformations, many such problems
take the form
\[
\mu := \int_{[0,1]^d} f(\bx) \, d\bx,
\]
where the dimension \( d \) is often large. Classical numerical methods, such as Gaussian quadrature or trapezoidal integration, can be highly efficient in low-dimensional settings but suffer from the curse of dimensionality, with the number of required points growing exponentially in \( d \). 

\emph{Monte Carlo (MC) methods} \cite{Lem09} avoid this by sampling \( n \) independent and identically distributed (IID) points \( \bx_1, \ldots, \bx_n \) uniformly from \([0,1]^d\) and approximate the integral $\mu$ as
\[
\hat{\mu}_n := \frac{1}{n} \sum_{i=1}^n f(\bx_i).
\]
The error of this estimator decays like $\sigma(f)/\sqrt{n}$, independent of the dimension $d$,
where \( \sigma(f) \) denotes the standard deviation of \( f \). 

\emph{Quasi-Monte Carlo (QMC) methods} \cite{HICKKIRKSOR2025, DicEtal14a, leopill14} aim to improve upon MC by replacing random samples with carefully constructed deterministic sampling nodes, known as low-discrepancy points sets and sequences, that are more evenly spread over \([0,1]^d\). The convergence rates of QMC methods are captured by one of several variants of the Koksma-Hlawka inequalities \cite{Kok42,Hla61}, which bounds the integration error as
\begin{align} \label{eq:koksma-hlawka}
|\hat{\mu}_n - \mu| \leq D(\{\bx_i\}_{i=1}^n) \cdot V_{\HK}(f),
\end{align}
where \( V_{\HK}(f) \) is the variation of \( f \) in the sense of Hardy and Krause \cite{Owe05a}, and \( D(\cdot) \) defined as 
\begin{equation}\label{eq:star_disc}
    D(\{
    \bx_i\}_{i=1}^n) := \sup_{\ba \in [0,1]^d} \left| \frac{1}{n} \sum_{i=1}^n \boldsymbol{1}\{\bx_i \in [\boldsymbol{0},\ba)\} - |[\boldsymbol{0},\ba)| \right|
\end{equation}
measures the uniformity of the sampling nodes with respect to axis-aligned origin-anchored test boxes, and is referred to as the {\em star discrepancy.}

Over the years, numerous families of low-discrepancy point sets have been proposed, with the two main families being digital nets and sequences \cite{Nie92, DP10} and lattice rules \cite{SloJoe94,DicEtal22a}, with more recent approaches leveraging optimization and machine learning techniques to minimize the star discrepancy \cite{ruschkirk24,cle25}. For a given low-discrepancy point set $\{\bx_i\}_{i=1}^n$, one can bound its star discrepancy by $C(\log n)^{d-1}/n$ for some constant $C$ depending on the specific construction. This implies a deterministic integration error of the same order whenever \( V_{\HK}(f) \) is finite. Compared with the MC error rate $O(n^{-1/2})$, the dependence on $n$ is asymptotically superior provided the dimension $d$ is not too large.

\subsection{Beyond Hardy-Krause via Combinatorial Discrepancy} 
In practice, the variation \( V_{\HK}(f) \) can be orders of magnitude larger than $\sigma(f)$, especially for functions with oscillatory or high-frequency components. Worse, the classical Koksma-Hlawka \eqref{eq:koksma-hlawka} bound ties QMC performance to \( V_{\HK}(f) \), so even implementing sampling node set with low-discrepancy may yield a pessimistic bound. Furthermore, classical QMC requires a deterministic node construction and cannot be readily applied when only random samples are available.

To address these issues, recent work by Bansal and Jiang \cite{BJ25} gives a randomized QMC construction that offers significantly tighter error bounds---they replace $V_{\HK}(f)$ in \eqref{eq:koksma-hlawka} with a ``smoothed-out variation'', $\sigma_{\SO}(f)$, which is substantially smaller for many practical integrands, especially those with substantial high-frequency component. 
Their QMC method is based on the notion of {\em combinatorial discrepancy}, that quantifies the necessary imbalance in any $\pm 1$ coloring of a finite point set with respect to a family of subsets \( \mathcal{F} \subseteq 2^{[n]}\), and is formally defined as
\[
\operatorname{disc}(\{ \bx_i \}_{i=1}^n, \mathcal{F}) = \min_{\chi: \{ \bx_i \}_{i=1}^n\to \{\pm 1\}} \max_{F \in \mathcal{F}} \Bigg| \sum_{i \in F} \chi(\bx_i) \Bigg|.
\]
Despite the different-looking definition, combinatorial discrepancy is known to be closely connected to star discrepancy through the {\em transference principle} (e.g., see \cite{Mat09,AistleitnerBN17}) --- which, in the QMC context, states that low-discrepancy point sets can be obtained from colorings with small combinatorial discrepancy with respect to anchored boxes (see Lemma \ref{lem:transference}, \Cref{subsubsec:disc_transference_subg}). 
Leveraging the latest algorithmic developments in combinatorial discrepancy theory in the transference principle, \cite{BJ25} obtained their QMC construction with the improved error bounds.

\subsection{Quasi-Monte Carlo in High Dimension}
\label{sec:qmc_high_dim}

The discussion so far has assumed fixed $d$ and sufficiently large \( n \), where QMC is known to beat MC in convergence rate. In high dimensions, however, the discrepancy bound \( (\log n)^d/n \) grows rapidly with $d$, making the QMC error bound substantially worse than the MC rate of \( O(n^{-1/2}) \);
see \cite[Section 5.3]{HICKKIRKSOR2025}. This led to the long-held belief that QMC was unsuitable for moderate- or high-dimensional problems (e.g., $d>10$). Surprisingly, empirical results from the 1990s in computational finance contradicted this view: studies such as \cite{PasTra95,PapTra97} reported QMC significantly outperforming MC for integrals with $d$ in the hundreds.

\smallskip
{\bf Effective Dimension and Weighted Function Spaces.} To explain this, the concept of \emph{effective dimension} emerged \cite{PasTra95,CafMor96,paskov97}, recognizing that many high-dimensional integrands depend strongly only on a few variables or low-order interactions. 
They have been formalized into two different notions of effective dimension---in the \emph{truncation} sense and in the \emph{superposition} sense (see \Cref{sec:effective_dimension} for the formal definitions). 

Around the same time, a rigorous framework arose from weighted function spaces \cite{SloWoz98}, which build variable importance directly into the analysis. In years since, many works exist in this setting, e.g., \cite{DicEtal06,SLOAN2001697,GNEWUCH201429}, where a collection of non-negative weights $\gamma_u$, one for each subset $u \subseteq \{1,2,\ldots,d\}$, quantify the relative importance of variable subsets. Larger $\gamma_u$ allow greater variation in that projection, so uniformity in sampling locations matters more there.  Since specifying all $2^d-1$ weights is infeasible for large $d$, they are typically parametrized. The most common choice of parametrization is that of product weights where each coordinate $1\leq j\leq d$ is assigned a weight $\gamma_j>0$ and $\gamma_u = \prod_{j\in u} \gamma_j$. It is usual to assume $\gamma_1 \geq \gamma_2 \geq \cdots \gamma_d \geq 0$, and tractability theory shows that under suitable decay of $\gamma_j$, QMC can achieve convergence rates independent of $d$ \cite{Sloan2002}.

\smallskip
\noindent \textbf{Existing High Dimensional QMC Constructions.} Many strategies have been developed for constructing QMC sampling nodes in high dimensions. Among the most widely used are digital nets and sequences \cite{DP10}, which use number-theoretic constructions to achieve uniformity. However, in high dimensions, these often suffer from correlations in low-dimensional projections \cite{kirklem24}. A common and effective workaround is to permute coordinates, randomly or deterministically \cite{Owe95, kirklem24, Lem04a, LEcLem02a}, which disrupt such correlations. This general method is limited and by no means optimal, particularly if something is known about the integrand. Another widely used class is lattice rules, fully determined by a generating vector \( \boldsymbol{h} \in \{0,1,\ldots,n-1\}^d \). The leading construction algorithm remains for many years the component-by-component (CBC) method \cite{Kro59, SloRez01} and can be implemented efficiently via fast Fourier transforms (FFTs) \cite{NuyCoo06b, NuyCoo06a, LatNet}. CBC readily incorporates coordinate weights guiding relative importance of coordinate subsets. Similar weighted constructions are also available for generating matrices in digital nets.

\subsection{Our Contributions}

Following the work of Bansal and Jiang \cite{BJ25}, we study new constructions of high dimensional QMC sets via combinatorial discrepancy theory. 
In particular, we adapt their constructions for integrands belonging to weighted function spaces and for functions with low effective dimension. 

We prove the following error bound for weighted function spaces, where the QMC point set is denoted as $A_T$ (see \Cref{sec:algorithm} for the construction). Here, for explicitness, we assume that the integrand $f$ has a Fourier decomposition $f(\bx) = \sum_{\bk \in \mathbb{Z}^d} \widehat{f}(\bk) \exp(2 \pi i \langle \bk, \bx \rangle)$.\footnote{As noted in \cite{BJ25}, it is possible to extend the notion of smoothed-out variation and the error bound to more general functions, but it requires taking certain limits and does not give clean formulas.}

\begin{figure}[t]
    \centering
    \includegraphics[width=0.7\linewidth]{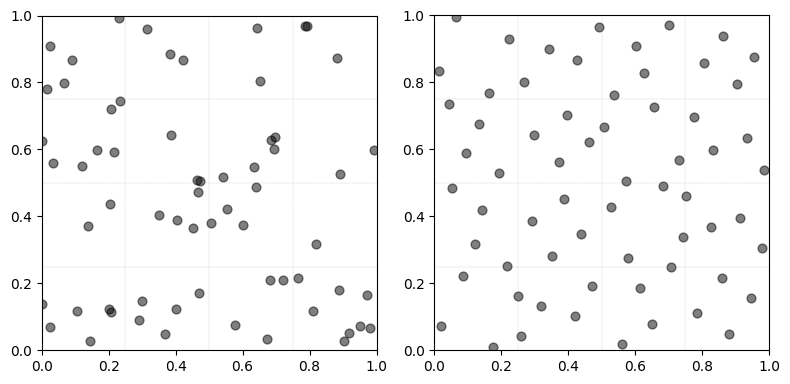}
    \caption{(\textbf{Left}) Uniform IID points  (\textbf{Right}) Points $A_T$ resulting from  algorithm $\SubgTrans$, $n=64$.}
    \label{fig:points}
\end{figure}

\begin{restatable}[Weighted subgaussian transference]{theorem}{WSubgTrans}\label{thm:main1}
    For every function $f:[0,1]^d \to \R$ that has a Fourier decomposition, the integration error satisfies 
    \[
    \E[(\err(A_T,f))^2]  \leq  O\left(\frac{d\log (dn)\prod_{j\in [d]}(1+\gamma_j^2\log n)}{n^2} \cdot \sigma_{\WSO}(f)^2 \right),
    \]
    where the {\em weighted smoothed-out variation}
    \begin{align}\label{eq:wso_f}
   \sigma_{\WSO}(f)^2:= \sum_{\bk\neq \mathbf{0}} \big|\widehat{f}(\bk) \big|^2 \left(\sum_{\emptyset \neq u \subseteq [d]}   \prod_{j \in u} \frac{|k_j|}{\gamma_j^2} \right).
    \end{align}
\end{restatable}

To the best of our knowledge, this is the first explicitly stated error bound for high dimensional QMC constructions. 
One can, in principle, obtain an error bound by applying (a general version of) Koksma-Hlawka inequality to weighted function spaces, but the error bound in \Cref{thm:main1} is favorable in comparison, as the weighted smoothed-out variation is substantially smaller than the Hardy-Krause variation of the weighted integrand \cite{BJ25}. 

We can obtain similar error bounds for functions with low effective dimensions, but we defer the formal statements to \Cref{sec:effective_dimension}. 
We also remark that as in \cite{BJ25}, the error rate of these QMC constructions is at most that of the MC error (up to a logarithmic factor), thereby achieving the best of both the error in \Cref{thm:main1} and the MC error.

\noindent\textbf{Numerical Experiments.} We also implement the subgaussian transference algorithm, providing the first numerical evaluations of the QMC constructions in \cite{BJ25} and the weighted extensions introduced here. Our experiments evaluate the irregularity of distribution of the resulting point sets and integration tests in both truncation and superposition dimension settings. The results investigate the robustness of subgaussian transference across different initializations, its consistent improvement over IID Monte Carlo, and the benefits of tailoring the algorithm to truncation or superposition effective dimension structure.

\medskip
\noindent \textbf{Roadmap.} The structure of this paper is as follows. Section~\ref{sec:algorithm} introduces the weighted subgaussian transference algorithm, $\wSubgTrans$, and establishes its theoretical properties. Section~\ref{sec:effective_dimension} develops variants of the subgaussian transference algorithm tailored to functions with low effective dimension and analyzes their theoretical guarantees. The proofs of the main results for the weighted setting (Section~\ref{sec:algorithm}) and the effective-dimension variants (Section~\ref{sec:effective_dimension}) are provided in Appendix~\ref{app:proof_1} and Appendix~\ref{app:proof_2}, respectively. Section~\ref{sec:numerics} reports numerical experiments considering both irregularity of distribution and high-dimensional integration approximation examples. Finally, Section~\ref{sec:discussion} discusses the implications of our findings and outlines directions for future work.

\section{The QMC Construction for Weighted Spaces}\label{sec:algorithm}

In this section, we present a variant of the subgaussian transference algorithm from \cite{BJ25}, adapted to weighted function spaces, and analyze its theoretical properties. We refer to this algorithm as $\wSubgTrans$ throughout. Subsection~\ref{subsec:notation} introduces the necessary notation, and Subsection~\ref{subsec:algorithm} describes the algorithm in detail.

\subsection{Notation}\label{subsec:notation}

For integer $d>0$, we write $[d] :=\{1, \ldots, d\}$. 
We use bold lowercase letters, e.g., $\bz=(z_1,\ldots,z_d)$, to denote vectors, with $z_j$ denoting the $j^{\mathrm{th}}$ coordinate of $\bz$. For a collection of vectors, we denote the $i^{\mathrm{th}}$ vector by $\bz_i$, and refer to its $k^{\mathrm{th}}$ coordinate as $z_{i,k}$. Given a function $f: [0,1]^d \rightarrow \R$ and a finite set of points $A\subseteq [0,1]^d$, we define the (signed) integration error with respect to point set $A$.
\[
\err(A,f):=\frac{1}{|A|}\sum_{\bx \in A} f(\bx)- \int_{[0,1]^d} f(\bx) \, d\bx.
\]

\smallskip
\noindent \textbf{Dyadic Intervals and Boxes.} 
Let $j \in \mathbb{N}$, we denote by $\D_j$ the set of all left-open dyadic intervals in $(0,1]$ of length $2^{-j}$, i.e., $\D_j := \{(\ell/2^j, (\ell+1)/2^j]: \ell = 0,\ldots, 2^j-1 \}$, and call $j$ the level of these dyadic intervals. For $h \in \mathbb{N}$, define $\D_{\leq h} := \bigcup_{j \leq h} \D_j$ as the set of all dyadic intervals of level less than or equal to $h$. In higher dimensions, $\D_{\leq h}^{\otimes d} := \{I_1 \times \cdots \times I_d: I_j \in \D_{\leq h} \text{ for all $j \in [d]$}\}$ denotes all $d-$dimensional dyadic boxes with level at most $h$ in each dimension.

\smallskip
\noindent \textbf{Product Weights.}  Let $\gamma_1\ge \gamma_2\ge \cdots \ge \gamma_d>0$ be a sequence of positive weights assigned to each coordinate direction as in \Cref{sec:qmc_high_dim}. For any dyadic box $B=B_1\times \cdots \times B_d\in \D_{\leq h}^{\otimes d}$, we say $B$ is \emph{trivial} in dimension $j$ if $B_j=(0,1]$.  The \emph{weight} of $B$, denoted as $\gamma(B)$, is then defined as the product of the weights corresponding to its nontrivial coordinate directions:
\begin{align}\label{eq:box_weight} 
\gamma(B):=\prod_{j\in [d]: B_j\ne (0,1]} \gamma_j,
\end{align}
with the convention that $\prod_{j\in \emptyset}\gamma_j=1$. For nonempty subset $u \subseteq [d]$, we slightly abuse notation and write $\gamma(u):=\prod_{j\in u}\gamma_j$. Lastly, throughout, we assume without loss of generality that $n$ is an integer power of $2$.

\subsection{$\wSubgTrans$ Algorithm}\label{subsec:algorithm}

The $\wSubgTrans$ algorithm aims to produce $n$ point sets, each of size $n$, with low continuous discrepancy, while accounting for product weights across coordinate directions. The algorithm begins with $n^2$ independent random samples and iteratively refines the point sets by applying balanced colorings that minimize the {\em combinatorial} discrepancy in an associated vector balancing problem. This combinatorial discrepancy minimization is performed via the $\SubgVecBal$ algorithm \cite{BJ25, ALS21} over $\log_2 n$ steps. At each step, the current sets are partitioned into halves, ultimately yielding $n$ subsets, each of cardinality $n$.

The vector balancing problem underlying this procedure is as follows: given vectors $\bv_1, \ldots, \bv_n \in \R^m$ of $\ell_2$-norm at most $1$, and the goal is to find a coloring $\bc=(c_1,\ldots,c_n) \in \{\pm 1\}^n$ to minimize the combinatorial discrepancy $\|\sum_{j=1}^n c_j \bv_j\|_\infty$. To address this, we first recall the \textsf{Self-Balancing Walk} algorithm from \cite{ALS21}, which solves the problem in an online manner by sequentially assigning signs to the input vectors while ensuring low discrepancy with probability at least $1 - \delta$.

\begin{algorithm}[H]
\caption{\textsf{Self-Balancing Walk} \cite{ALS21}} 
\label{alg:balance}
\begin{algorithmic}[1]
\State $\bw_0 \gets 0$
\State $\lambda \gets 30 \log(mn/\delta)$
\For{$1 \leq j \leq n$}
    \If{$|\langle \bw_{j-1}, \bv_j\rangle| > \lambda$ \textbf{ or } $\|w_{j-1}\|_\infty > \lambda$}
        \State \textbf{Fail} \Comment{Algorithm terminates with failure}
    \EndIf
    \State $p_j \gets \frac12 - \frac{\langle \bw_{j-1}, \bv_j\rangle}{2\lambda}$
    \State $c_j \gets 1$ with probability $p_j$, and $c_j \gets -1$ with probability $1-p_j$
    \State $\bw_j \gets \bw_{j-1} + c_j \bv_j$
\EndFor
\State \Return coloring $\bc=(c_1,\ldots,c_n)\in \{\pm 1\}^n$ 
\end{algorithmic}
\end{algorithm}

The idea behind \textsf{Self-Balancing Walk} is as follows. At step $j$, the algorithm maintains the partial sum $\bw_{j-1}$ and decides the sign of $\bv_j$ probabilistically depending on $\langle \bw_{j-1},\bv_{j}\rangle$. If $\bw_{j-1}$ and
$\bv_j$ are orthogonal, it chooses $c_j=\pm 1$ with equal probability. As the alignment $\langle \bw_{j-1},\bv_{j}\rangle$ increases, the probability of selecting $c_j=-1$ increases, counteracting drift in that direction. The parameter $\lambda=30\log(mn/\delta)$ ensures that, with probability at least $1-\delta$,
$|\langle \bw_{j-1}, \bv_j\rangle|,\, \|\bw_{j-1}\|_\infty \le \lambda$, so the iterations remain feasible with high probability. A formal performance guarantee is given in Theorem~\ref{thm:self_bal_walk} (Appendix~\ref{app:proof_1}).

To ensure that the coloring $\bc=(c_1,\ldots,c_n)$ is {\em balanced}, i.e., has equal number of $1$'s and $-1$'s,  \cite{BJ25} proposed a modified algorithm $\SubgVecBal$. For even $n$, this is achieved by pairing adjacent vectors to form $\bv_1 - \bv_2, \bv_3 - \bv_4, \cdots, \bv_{n-1} - \bv_n$, and then applying \textsf{Self-Balancing Walk} to the resulting set. We will assume the balanced property henceforth.

Our $\wSubgTrans$ algorithm, which constructs product-weighted low-discrepancy point sets, is formally defined as follows:
\begin{algorithm}[H]
\caption{\(\wSubgTrans\)}
\label{alg:wSubgTrans}
\begin{algorithmic}[1]
\State $h \gets O(\log (dn))$
\State Sample a random shift $\boldsymbol{s} \sim [0,1)^d$ and shift the dyadic system $\mathcal{D}_{\leq h}^{\otimes d}$ accordingly (fold coordinates exceeding $1$ back into $[0,1)$). \Comment{For exposition, take $\boldsymbol{s} = \boldsymbol{0}$}
\State $n_0 \gets n^2$, \quad $T \gets \log_2 n$
\State $A_0^{(0)} \gets$ $n_0$ IID\ uniformly distributed samples from $[0,1]^d$
\For{$t = 0$ \textbf{to} $T-1$}
    \For{$i = 0$ \textbf{to} $2^t-1$}
        \State $A_t \gets A_t^{(i)}$
        \For{each $\bz_j \in A_t$}
            \State $\bv_j \gets \Big( \be^{A_t}_j , \big(\gamma(B)\ind_{\{\bz_j \in B\}}\big)_{B \in \D_{\leq h}^{\otimes d}} \Big)$ \Comment{$\be^{A_t}_j \in \{0,1\}^{|A_t|}$, with a single one in the $j$th coordinate}
        \EndFor
        \State $\bc_t \gets \SubgVecBal (\{\bv_j\})$ 
        \Comment{ $\bc_t \in \{\pm 1\}^{|A_t|}$: balanced coloring}
        \State $A_{t+1}^{(2i)} \gets \{\bz_j \in A_t : c_{t,j} = -1\}$
        \State $A_{t+1}^{(2i+1)} \gets \{\bz_j \in A_t : c_{t,j} = +1\}$
    \EndFor
\EndFor
\State \Return $\{A_T^{(i)} : 0 \le i < 2^{T}\}$ \Comment{$2^{T}(=n)$ final sets, each of size $n$}
\end{algorithmic}
\end{algorithm}

Our theoretical result provides the following guarantee for the point sets $A_T$ generated by the $\wSubgTrans$ algorithm; the proof is given in Appendix \ref{app:proof_1}.

\WSubgTrans*

\begin{remark}
The $\wSubgTrans$ algorithm depends on the weights $\gamma_1\ge \gamma_2\ge \cdots\ge \gamma_d>0$ assigned to coordinate directions. In the unweighted case ($\gamma_j\equiv 1$), it reduces to the subgaussian transference algorithm of \cite{BJ25}, hereafter denoted $\SubgTrans$. Under this specialization, Theorem \ref{thm:main1} recovers \cite[Theorem 1.1]{BJ25}, yielding the bound $\frac{(\log n)^{O(d)}}{n^2}\sigma_{\SO}(f)^2$, where $\sigma_{\SO}(f)^2$ (``smoothed-out variation'') coincides with $\sigma_{\WSO}(f)^2$ when $\gamma_j\equiv 1$. This bound becomes vacuous in high dimensions due to the $(\log n)^{O(d)}$ factor. By contrast, for product weights $\gamma_j=j^{-\alpha}$ with some $\alpha>1/2$,
    \begin{align*}
    \prod_{j\in [d]}(1+\gamma_j^2\log n)=\prod_{j\in [d]}\bigg(1+\frac{\log n}{j^{2\alpha}}\bigg)\le \exp\big(C_{\alpha}(\log n)^{\frac{1}{2\alpha}}\big),
    \end{align*}
    where $C_{\alpha}=2\alpha+1+\frac{1}{2\alpha-1}$. In particular, the right-hand side is dimension-independent and subpolynomial in $n$: for any $\varepsilon>0$, $ \exp\big(C_{\alpha}(\log n)^{\frac{1}{2\alpha}}\big)=o(n^{\varepsilon})$ as $n\to \infty$.
\end{remark}

\section{The Constructions for Functions with Low Effective Dimension}\label{sec:effective_dimension}

As discussed in the introduction, many high-dimensional functions encountered in practice can be well-approximated by depending strongly on only a few coordinates, or on low-order interactions between them. This motivates the notion of effective dimension, which captures the idea that the function’s relevant dimensionality is much smaller than its nominal dimension $d$.

In this section, we introduce two variants of the $\wSubgTrans$ algorithm that explicitly incorporate effective dimension. These variants correspond to two standard measures: the \emph{superposition sense}, which quantifies the importance of variable interaction order, and the \emph{truncation sense}, which quantifies the importance of the earliest coordinates. We begin with a brief review of the analysis of variance (ANOVA) decomposition and the associated definitions of effective dimension \cite[Appendix A]{Owe13}. Subsection \ref{subsec:superposition} addresses the superposition dimension, while Subsection \ref{subsec:truncation} focuses on the truncation dimension.

For a subset of coordinates $u \subseteq [d]$, we denote its cardinality by $|u|$, and write $\overline{u}:= [d] \setminus u$ for its complement. Given a vector $\bx \in \R^d$ and a subset $u \subseteq [d]$, we denote by $\bx^{u} \in \R^{|u|}$ the restriction of $\bx$ to the coordinates indexed by $u$. In the analysis of variance (ANOVA) decomposition, a function $f$ can be expressed as
\begin{align*}
f(\bx)=\sum_{u\subseteq [d]}  f_u (\bx),
\end{align*}
where each component $f_u(\cdot)$ depends only on the variables $\bx^{u}$. More precisely, for any subset $u\subseteq [d]$, the corresponding term is defined by
\[
f_u(\bx):=\int_{[0,1]^{|\overline{u}|}} \Big(f(\bx)-\sum_{v\subsetneq u} f_{v}(\bx)\Big) d\bx^{\overline{u}}
\]
The zero-order term $f_{\emptyset}(x)$ is a constant that is equal to the mean of $f$. The ANOVA decomposition is orthogonal in $L^2$, so that 
\[
\sigma(f)^2=\sum_{u\subseteq [d] } \sigma_u^2,\quad  \text{where} \ \ \sigma_u^2:= \sigma(f_u)^2.
\]

We now recall the definitions of two notions of effective dimension in the superposition sense and in the truncation sense following \cite{CafMorOwe97}.

\begin{definition}[Effective dimension, superposition sense]\label{def:superposition}
    The effective dimension of $f$, in the superposition sense, is the smallest integer $d_S$ such that $\sum_{0<|u|\le d_{S}} \sigma (f_u)^2\ge 0.99 \sigma(f)^2$.
\end{definition}

\begin{definition}[Effective dimension, truncation sense]\label{def:truncation}
    The effective dimension of $f$, in the truncation sense, is the smallest integer $d_T$ such that $\sum_{u\subseteq \{1,2, \ldots, d_T\} } \sigma (f_u)^2\ge 0.99 \sigma(f)^2$.
\end{definition}

\subsection{Effective Superposition Dimension}\label{subsec:superposition}

For $h \in \mathbb{N}$ and $s\in [d]$, let $\widetilde{\D}_{\leq h}^{\otimes s}$ be the collection of dyadic boxes of level at most $h$ that involve at most $s$ non-trivial dimensions; this is a subset of the full dyadic collection $\D_{\leq h}^{\otimes d}$ defined in Subsection \ref{subsec:notation}. For functions with small effective dimension in the superposition sense (Definition \ref{def:superposition}), we modify the $\wSubgTrans$ algorithm by forming the incidence vector of $\bz_j$ with respect to $\widetilde{\D}_{\leq h}^{\otimes s}$ rather than the full collection $\D_{\leq h}^{\otimes d}$. Specifically, on line $9$ of Algorithm \ref{alg:wSubgTrans}, we define the vector $\bv_j$ as
\begin{align} \label{eq:stacked_vector}
\bv_j = \Big( \be^{A_t}_j , \big(\ind_{\{\bz_j \in B\}}\big)_{B \in \widetilde{\D}_{\leq h}^{\otimes s}} \Big) ,
\end{align}
where $\ind_{\{\bz_j \in B\}}=1$ if $\bz_j \in B$ and $0$ otherwise. All other steps of the $\wSubgTrans$ algorithm remain unchanged. The modified algorithm outputs $n$ sets of points, each of equal size $n$; we use $A_{T}$ to denote any one of these sets.

Our next result provides a performance guarantee for the point set $A_T$ generated by the modified $\wSubgTrans$ algorithm. The proof is given in Appendix \ref{app:proof_2}.

\begin{theorem}\label{thm:superposition}
      For every function $f:[0,1]^d \to \R$ that has a Fourier decomposition, the integration error satisfies\footnote{$O_s(1)$ denotes a constant that depends only on $s$.}
    \[
    \E[(\err(A_T,f))^2] \leq O_s(1) \cdot \inf_{f = g + h,\, h\in \F_{s}} \bigg( \frac{\log n}{n}\sigma(g)^2 + \frac{d^s(\log (dn))^{s+1}}{n^2}\sigma_{\SO}(h)^2 \bigg),
    \]
    where the infimum is over all decompositions of $f$ as the sum of two functions $g, h :[0,1]^d \to \R$ with $g\in L^2([0,1]^d)$ and $h\in \F_{s}$. The class $\F_{s}$ consists of functions involving interactions of order at most $s$, namely 
    \[
    \F_{s}=\bigg\{\sum_{|u|\le s} f_{u}(\bx^u): f_u\in L^2([0,1]^d)\bigg\}.
    \]
    The {\em smoothed-out variation} is defined as
    \begin{align*}\label{eq:so_f}
   \sigma_{\SO}(h)^2:= \sum_{\bk\neq \mathbf{0}} \big|\widehat{h}(\bk) \big|^2 \bigg(\sum_{\emptyset \neq u \subseteq [d]}   \prod_{j \in u} |k_j|\bigg).
    \end{align*}
\end{theorem}

\begin{remark}
      In particular, for any function $f:[0,1]^d \to \R$ with effective dimension $s$ in the superposition sense, consider its ANOVA decomposition 
    \[
    f(\bx)=\sum_{|u|\le s}  f_u (\bx)+\sum_{|u|>s}  f_u (\bx)=:f_{\le s}+f_{> s}.
    \]
    Then the integration error satisfies
    \[
    \E[(\err(A_T,f))^2] \leq O_s\bigg( \frac{\log n}{n}\sigma(f_{>s})^2 + \frac{d^s(\log (dn))^{s+1}}{n^2}\sigma_{\SO}(f_{\le s})^2 \bigg).
    \]
    Since $\sigma(f_{\le s})^2\ge 0.99 \sigma(f)^2$ and $\sigma(f_{>s})^2\le 0.01 \sigma(f)^2$, the modified $\wSubgTrans$ algorithm is expected to perform well in high dimensions for functions with small effective superposition dimension. 
    
    Moreover, the number of nonzero entries in the vector $\bv_j$ is $O_{s}(d^s h^s)=O_{s}((d\log (dn))^s)$, leading to significantly lower computational cost compared to the full 
$\wSubgTrans$ algorithm described in Section~\ref{sec:algorithm}. Specifically, the full $\wSubgTrans$ algorithm has an overall runtime of
\[
O\left(\,\sum_{t=0}^{T-1}\frac{n^2}{2^{t}}\cdot h^{d}\right)=O\big(n^2 (\log (dn))^d\big),
\]
while the modified algorithm tailored to superposition of order $s$ achieves an overall runtime of
\[
O_s\left(\,\sum_{t=0}^{T-1}\frac{n^2}{2^{t}}\cdot (d\log (dn))^s\right)=O_s\big(n^2 (d\log (dn))^s\big).
\]
The corresponding amortized time per output point is $O\big((\log (dn))^d\big)$ for the full algorithm and $O_s\big((d\log (dn))^s\big)$ for the modified version. When the dimension $d$ is large and the effective superposition dimension $s$ remains moderate, this results in a substantial computational improvement. 
\end{remark}

\subsection{Effective Truncation Dimension}\label{subsec:truncation}

 In this subsection, we modify the $\wSubgTrans$ algorithm to accommodate functions with small effective dimension in the truncation sense (Definition~\ref{def:truncation}). This setting can be viewed as a special case of the general $\wSubgTrans$ algorithm presented in Section~\ref{sec:algorithm}, where the weights are chosen as $\gamma_1=\cdots =\gamma_s=1$ and $\gamma_{s+1}=\cdots=\gamma_{d}=0$ for a fixed integer $s\in [d]$. In this case, 
\[
\|\bv_j\|_2^2=1+\prod_{j\in [d]}(1+h\gamma_j^2)=O(h^{s})=O\big((\log (dn))^s\big).
\]
The overall runtime of the algorithm with this specific choice of weights is $O\big(n^2 (\log (dn))^s \big)$,  yielding an amortized time per output point of $O\big( (\log (dn))^{s} \big)$.

Our next result provides a performance guarantee for the point set $A_T$ generated by the modified $\wSubgTrans$ algorithm. The proof is given in Appendix \ref{app:proof_2}.

\begin{theorem}\label{thm:truncation}
   For every function $f:[0,1]^d \to \R$ that has a Fourier decomposition, the integration error satisfies
    \[
    \E[(\err(A_T,f))^2] \leq O_s(1) \cdot   \inf_{f = g + h,\, h\in \widetilde{\F}_{s} } \bigg( \frac{\log n }{n}\sigma(g)^2 + \frac{(\log (dn))^{s+1}}{n^2}\sigma_{\SO}(h)^2 \bigg),
    \]
      where the infimum is over all decompositions of $f$ as the sum of two functions $g, h :[0,1]^d \to \R$ with $g\in L^2([0,1]^d)$ arbitrary and $h\in \widetilde{\F}_{s}$. The class $\widetilde{\F}_{s}$ consists of functions that depend only on the variables $\bx^{\{1,2,\ldots,s\}}$.
\end{theorem}

\begin{remark}
For any function $f:[0,1]^d \to \R$ with effective dimension $s$ in the truncation sense, consider its ANOVA decomposition 
    \[
    f(\bx)=\sum_{u\subseteq \{1,2,\ldots,s\}}  f_u (\bx)+\sum_{u\nsubseteq \{1,2,\ldots,s\}}  f_u (\bx)=:f^{(1)}_{s}+f^{(2)}_{s}.
    \]
    Then the integration error satisfies
    \[
    \E[(\err(A_T,f))^2] \leq  O_s \bigg( \frac{\log n}{n}\sigma(f^{(2)}_{s})^2 + \frac{(\log (d n))^{s+1}}{n^2}\sigma_{\SO}(f^{(1)}_{s})^2 \bigg).
    \]
    Since $\sigma(f^{(1)}_{s})^2\ge 0.99 \sigma(f)^2$ and $\sigma(f^{(2)}_{s})^2\le 0.01 \sigma(f)^2$, the modified $\wSubgTrans$ algorithm is expected to perform well in high dimensions for functions with small effective truncation dimension. 
    \end{remark}

\section{Numerical Experiments}\label{sec:numerics}

We evaluate four variants of the Subgaussian Transference algorithm: the original formulation from \cite{BJ25}, Algorithm \ref{alg:wSubgTrans} from Section \ref{sec:algorithm}, and two new variants designed to account for effective dimension in the truncation and superposition senses. Our numerical study examines both the uniformity of the resulting point sets and their performance in integral approximation.

\subsection{Implementation Detail}

To reduce the algorithm runtime, several practical modifications were introduced in our implementations. First, the parameter $\lambda$ was set sufficiently small (e.g., $\lambda \approx 10^{-3}$) in Algorithm \ref{alg:balance}, and instead of failing, the algorithm chooses the sign $c_j$ greedily when $|\langle \bw_{j-1}, \bv_j\rangle| > \lambda$.
Second, instead of drawing $n^{2}$ candidate points, we use a reduced population of size $k n$, with $k$ a power of two (usually $k=16$). This lowers both time and memory complexity to $O(n)$. When $k=n$, we recover the precise oversampling procedure in the algorithm. When an effective dimension $s$ is specified, the dyadic refinement parameter is set to $h = \lceil \log_{2}(sn) \rceil$ rather than the nominal choice $h = \lceil \log_{2}(dn) \rceil$, capping the exponential growth of dyadic boxes in high dimensions. Lastly, we experiment with initializing the algorithm with the first $kn$ Sobol' points, rather than uniform IID random points to investigate the effect of the values of the starting set.

These cost-saving adjustments do come at the expense of some algorithmic precision—most notably due to the coarser dyadic refinement and smaller oversampling initialization. However, as we will see, they do not compromise performance enough to prevent improvements over MC methods.

\subsection{Resulting Star Discrepancy}

\begin{table}[t]
\centering
\begin{tabular}{|r|c|c|c|c|c|c|}
\hline
$n$ & Sobol' & IID & ST (IID, $n^2$) & ST (IID, $16n$) & ST (Sobol', $n^2$) & ST (Sobol', $16n$) \\
\hline
  8  & 0.312500 & 0.354866 & 0.301768 & 0.325006 & 0.305120 & 0.318114 \\
 16  & 0.171875 & 0.266701 & 0.199798 & 0.190557 & 0.201280 & 0.206695 \\
 32  & 0.089844 & 0.182864 & 0.140638 & 0.149751 & 0.125465 & 0.120430 \\
 64  & 0.053711 & 0.148777 & 0.084015 & 0.093255 & 0.088080 & 0.084042 \\
128  & 0.025146 & 0.106092 & 0.053085 & 0.059950 & 0.053417 & 0.051763 \\
256 & 0.014587 & 0.082900 & 0.033081 & 0.034835 & 0.032951 & 0.032138 \\
\hline
\end{tabular}
\caption{Star discrepancy: Sobol' and IID baseline and $\SubgTrans$ mean (ST) with different initializations of distribution and population size.}
\label{tab:stardisc_mean}
\end{table}

We implement the original $\SubgTrans$ algorithm from \cite{BJ25} (or equivalently, $\wSubgTrans$ with unit weights) in two dimensions. For each $n\in \{8,16,32,64,128, 256\}$, we generate output sets from population sizes of $kn$ (i) $k=16$, and (ii) $k=n$ to test the algorithm for robustness. To assess the geometric uniformity of the resultant sets and to illustrate the transference principle (see Lemma \ref{lem:transference}), we compute the star discrepancy \eqref{eq:star_disc} for each output set and report the mean for each configuration. As reference, we also report the star discrepancy of the unscrambled Sobol' set and the average of $16$ independent IID uniform point sets of the same $n$.  

From the results contained in Table \ref{tab:stardisc_mean}, Sobol' generally attains the lowest star discrepancy for $n\geq16$. For example, at $n=128$, the Sobol' set reports a star discrepancy of $0.025$, whereas $\SubgTrans$ averages range from $0.052$ to $0.060$ depending on initialization. However, for $n=128$, we found in pur investigation that the minimum star discrepancy reported was $0.040697$ for Sobol' initialization and $k=n$, and $0.39367$ for IID initialization and $k=n$. This brings us to our key point: initialization matters; with Sobol' starts, $k=16$ tends to slightly beat the $k=n$ start for $n\geq32$; with uniform starts, the $k=n$ start tends to be better at larger $n$. 

We also note that $\SubgTrans$ with IID initialization is always superior in comparison to IID. This suggests that in applications where only IID samples are available (and thus Sobol' is not applicable), subgaussian transference provides a competitive alternative to MC.

\subsection{Small Truncation Dimension}

\begin{figure}
    \centering
    \includegraphics[width=0.66\linewidth]{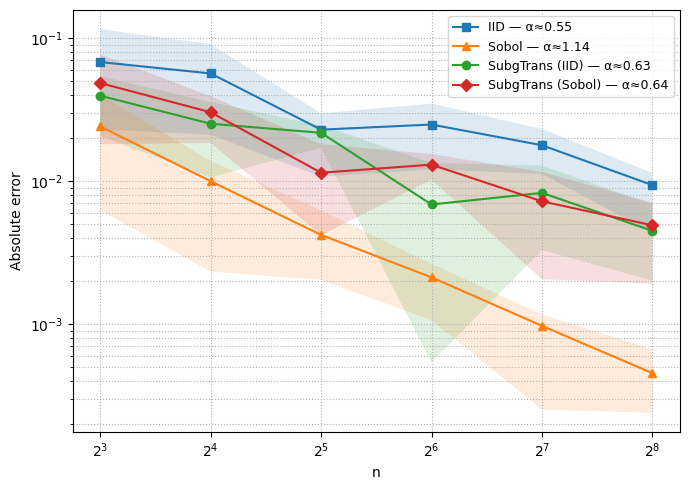}
    \caption{MAE with IQR for the approximation of the integral of function $f$ using $\wSubgTrans$ (with IID and Sobol' initializations), Sobol' and IID.}
    \label{fig:test_integral}
\end{figure}

We now move to the high-dimensional integration setting, restricting attention to problems of low effective dimension due to computational constraints. To test the truncation variant of $\wSubgTrans$ introduced in Section \ref{subsec:truncation}, we consider the function $$f(\boldsymbol{x})= \sum_{i=1}^d (-1)^i \prod_{j=1}^i x_j$$ which for nominal dimension $d=100$ is known to have a truncation effective dimension of $2$ \cite{KUCHERENKO2001}. The true integral value can be calculated analytically as $\frac{(1-(-1/2)^d)}{-3}$. We approximate this integral using sampling nodes generated by $\wSubgTrans$ with weight vector $(1,1,0,\ldots,0)$, alongside IID and scrambled Sobol’ baselines, and compute the absolute error to the known true value. The convergence rates are also reported in the form $\mathcal{O}(n^{-\alpha})$.

Figure \ref{fig:test_integral} shows that the Sobol’ sequence provides the strongest performance, with rapid decay of error and low variability across replications. Nevertheless, $\wSubgTrans$ delivers a clear improvement over IID sampling, achieving faster convergence ($\alpha \approx 0.65$ vs.\ $\alpha \approx 0.55$ for IID) and consistently smaller absolute errors. It seems initialization has little effect in this example: whether starting from IID or Sobol’ inputs, the partitioned sets exhibit nearly identical accuracy, indicating that the algorithm robustly adapts to the effective dimension structure. This highlights the central benefit of $\wSubgTrans$: even in very high nominal dimension ($d=100$), exploiting the effective dimension knowledge of a problem and restricting the action to the important variables yields point sets that outperform IID by imposing low-discrepancy structure where it matters.

\subsection{Small Superposition Dimension}

\begin{figure}
    \centering
    \includegraphics[width=0.66\linewidth]{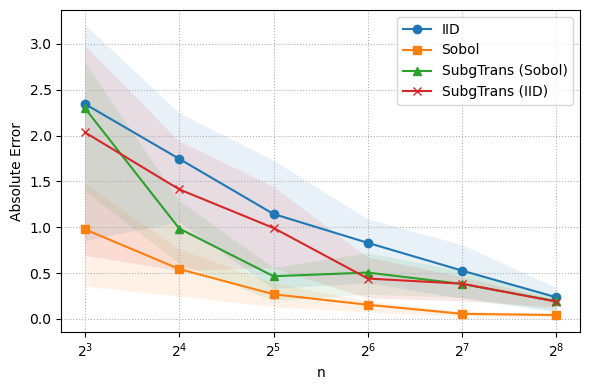}
    \caption{Approximation error of an Asian call option pricing problem using $\SubgTrans$ (with IID and Sobol' initializations), Sobol', and IID.}
    \label{fig:asian_results}
\end{figure}

As a benchmark application with low effective dimension, we consider pricing an Asian call option on a stock following geometric Brownian motion; see \cite{wangsloan05,lemowen01,faulem09}. The option payoff depends on the average of the asset price observed at monthly intervals over one year, leading to a $d=12$ dimensional integration problem. The present value is given by
\begin{equation}
C_0 = \mathbb{E} \left[ e^{-rT} \left( \frac{1}{d} \sum_{j=1}^d S(u_j) - K \right)^+ \right].
\end{equation}\label{eq:C0expectation}
where $S(u)$ is the stock price at time $u$ under the risk-neutral measure, $K$ is the strike, $T$ the maturity, and $r$ the risk-free rate. Assuming geometric Brownian motion with volatility $\sigma$, this expectation admits an integral representation over $[0,1]^d$ in terms of inverse normal transforms of uniform random variables \cite{faulem09}. Using parameters $S(0)=50$, $K=45$, $T=1$, $r=0.05$, and $\sigma=0.3$, the true option value is $C_0 = 7.2110915$, estimated in advance to high precision using 2M Sobol’ points.

Like many finance problems, this integral has a much lower effective dimension than its nominal value: the superposition effective dimension is approximately $2$, so most of the variance is explained by one- and two-way interactions among inputs. To exploit this, we apply the $\wSubgTrans$ algorithm with unit coordinate weights with $s=2$, i.e., the algorithm will only construct the incidence vector with $\D_{\leq h}^{\otimes 2}$, dyadic intervals upto an including order $2$. We generate sets of size $n \in \{8,16,32,64,128,256\}$ from initial populations of size $16n$. We again compare against IID and scrambled Sobol' points.

We see that as per the previous example, Sobol' still provides the sharpest convergence, but $\SubgTrans$ tracks much more closely to Sobol' for small and moderate $n$, narrowing the gap substantially relative to IID. The shaded error bands reveal that $\SubgTrans$ points also reduce variability across repetitions, producing more stable estimates than IID even at coarse sample sizes. Once again, we see that by partitioning the population set via subgaussian transference in a targeted manner to obtain sampling nodes, when something is known about the problem, one can achieve much better results than random sampling.

\section{Discussion}\label{sec:discussion}

In this work we extended the original subgaussian transference algorithm of \cite{BJ25} to the weighted setting, as is standard in modern QMC methods, and introduced effective-dimension variants targeting both truncation and superposition structure. On the theoretical side, we extended the error analysis in \cite{BJ25} to capture the weighted and low effective-dimension structure.
This provides error guarantees that depend explicitly on the chosen weights, and hence on the effective dimension of the integrand, rather than on the full Hardy–Krause variation as in the classical Koksma–Hlawka inequality. In high-dimensional settings, this provides more interpretable control.

The algorithms developed here are also related to the \emph{subset selection} problem \cite{cle22}, which seeks a single subset minimizing star discrepancy. In contrast, subgaussian transference partitions a large pool into many $n$-point sets, naturally providing replications that are useful in error approximation.

On the computational side, our modifications of reducing candidate pool size, adjusting refinement depth, and setting $\lambda$ small, kept runtimes tractable although paying a cost of accuracy at times. Generally speaking, scaling remains an issue: the exponential growth of dyadic boxes in the full algorithm makes naive application infeasible in very high dimensions. The effective-dimension variants alleviate this to some degree by reducing the combinatorial complexity, but further work is needed to bring the approach to genuinely high dimensional problems.

Overall, subgaussian transference provides theoretically grounded guarantees for integration in weighted spaces, alongside an algorithmic framework that produces replicated QMC point sets adapted to effective dimension. In all of the numerical experiments, subgaussian transference (with IID initializations) outperforms MC, providing a more favorable alternative in situations where only random samples are accessible.
However, we have also seen that it does not yet outperform standard QMC constructions in practice and further work is required to make the algorithm competitive.

\paragraph{Acknowledgments.} The third author is supported by NSF grant DMS-2316011.
The authors would like to thank Zexin Pan, Art B. Owen and Fred J. Hickernell for helpful discussions.

\bibliographystyle{alpha}
\bibliography{bib.bib,NMK25.bib}

\appendix
\section{Appendix}
\subsection{Proofs of results in Section \ref{sec:algorithm}}
\label{app:proof_1}

\smallskip
{\bf Notation.}  For integer $d>0$, we write $[d] :=\{1, \ldots, d\}$. 
We use bold lowercase letters, e.g., $\bz=(z_1,\ldots,z_d)$, to denote vectors, with $z_j$ denoting the $j^{\mathrm{th}}$ coordinate of $\bz$. For a collection of vectors, we denote the $j^{\mathrm{th}}$ vector by $\bz_j$, and refer to its $k^{\mathrm{th}}$ coordinate as $z_{j,k}$. For a subset of coordinates $u \subseteq [d]$, we denote its cardinality by $|u|$, and write $\overline{u}:= [d] \setminus u$ for its complement. Given a vector $\bx \in \R^d$ and a subset $u \subseteq [d]$, we denote by $\bx^{u} \in \R^{|u|}$ the restriction of $\bx$ to the coordinates indexed by $u$.

\smallskip
{\bf Dyadic Intervals.} Let $j \in \mathbb{N}$, we denote by $\D_j$ the set of all left-open dyadic intervals in $(0,1]$ of length $2^{-j}$, i.e., $\D_j := \{(\ell/2^j, (\ell+1)/2^j]: \ell = 0,\ldots, 2^j-1 \}$, and call $j$ the level of these dyadic intervals. For $h \in \mathbb{N}$, define $\D_{\leq h} := \bigcup_{j \leq h} \D_j$ as the set of all dyadic intervals of level less than or equal to $h$. In higher dimensions, $\D_{\leq h}^{\otimes d} := \{I_1 \times \cdots \times I_d: I_j \in \D_{\leq h} \text{ for all $j \in [d]$}\}$ denotes all $d-$dimensional dyadic boxes with level at most $h$ in each dimension. 

\smallskip
{\bf Prefix Intervals.} For any $x \in [0,1]$, we denote left-open prefix interval $C_{  x} := (0,x]$, with $C_1 = (0,1]$. For any $\bz=(z_1,\ldots,z_d) \in [0,1]^d$, we denote $C_{\bz} := (0,z_1] \times \cdots \times (0,z_d]$ the left-open corner at $\bz$.  We also denote $\mathcal{C}^d := \{C_{\bz}: \bz \in [0,1]^d\}$ the set of all left-open corners. For $h \in \mathbb{N}$, we use $\mathcal{C}_h$ to denote all left-open prefix intervals whose length is an integer multiple of $2^{-h}$ (excluding\footnote{This is again due to the slight technicality caused by the boundary.} $(0,1]$), i.e. $\mathcal{C}_h := \{C_{ \ell/2^h}: \ell = 0, \ldots, 2^h-1\}$.  We use $\mathcal{C}_h^{\otimes d} := \{I_1 \times \cdots \times I_d: I_j \in \mathcal{C}_h \text{ for all $j \in [d]$}\}$ to denote all left-open corners at grid points in $2^{-h} \Z^d \cap [0,1)^d$.  

\smallskip
{\bf Dyadic Decomposition.} {\em Dyadic decomposition} refers to decomposing sub-intervals of $(0,1]$ into a minimal disjoint union of  dyadic intervals. We only use dyadic decomposition for prefix intervals. 
For any $h \in \mathbb{N}$, we define $P_h \in \{0,1\}^{|\mathcal{C}_h| \times (|\D_{\leq h}|-1)}$ the dyadic decomposition matrix of $\mathcal{C}_h$, where the number of columns is $|\D_{\leq h}|-1$ because we have chosen to exclude the unused $(0,1] \in \D_{\leq h}$ in the decomposition. 
In particular, for the $\ell$th prefix interval $I_\ell \in \mathcal{C}_h$ and the $r$th dyadic interval $J_r \in \D_{\leq h} \setminus \{C_1\}$, we have $(P_h)_{\ell, r} = 1$ if $J_r$ is used in the dyadic decomposition of $I_\ell$, and $0$ otherwise.

As each prefix in $\mathcal{C}_{\leq h}$ is a sum of at most $h$ dyadic intervals in $\D_{\leq h}$, each row of $P_h$ has at most $h$ ones. However, the columns of $P_h$ could be dense with $\Omega(2^h)$ entries. For our analysis, we will in fact crucially exploit the structure of the columns of $P_h$. 
 For technical convenience and following \cite[Lemma 2.2]{BJ25}, we actually work with a slight variant of $P_h$ that has more structured columns, which we call the {\em structured decomposition matrix} and denote it as $\overline{P}_h\in \{0,1\}^{|\mathcal{C}_h| \times (|\D_{\leq h}|-1)}$. In higher dimension, the dyadic decomposition of $\mathcal{C}_h^{\otimes d}$ into dyadic boxes $\D_{\leq h}^{\otimes d}$ is given by $P_h^{\otimes d}$, and the structured decomposition tensor is given by $(\overline{P}_h)^{\otimes d}$.

\subsubsection{Combinatorial Discrepancy, Transference Principle, and Subgaussianity}
\label{subsubsec:disc_transference_subg}

Let us consider an iteration $t \in \{0,\ldots, T-1\}$ of the $\wSubgTrans$ algorithm where $A_{t+1}$ is produced from $A_t$ by running $\SubgVecBal$.
For any set $C \subseteq [0,1]^d$, let
\begin{align*}
\disc_t(C) :=  | \{\bz_j \in C \cap A_t \text{ with } c_{t,j} = +1\}| - | \{\bz_j \in C \cap A_t \text{ with } c_{t,j} = -1\}|
\end{align*}
denote the \emph{combinatorial discrepancy} of $C \cap A_t$. In particular, we denote $\disc_t(\bz) := \disc_t(C_{\bz})$ for any $\bz \in [0,1]^d$. For each $t \in \{0, \ldots, T\}$, we denote $n_t := |A_t| = 2^{-t}n_0$. 
For any measurable set $C \subseteq [0,1]^d$, 
denote $n_t(C) := |C \cap A_t|$, and \emph{continuous discrepancy} $h_t(C) := \vol(C) - n_t(C)/n_t$. We denote $h_t(\bz):=h_t(C_{\bz})$ for any $\bz\in [0,1]^d$.

Let \(\mathcal{I}:=\D_{\le h}^{\otimes d}\setminus\{C_{\ind}\}\) be the collection of left-open dyadic boxes up to level \(h\) in dimension \(d\), excluding the box \(C_{\ind}\).
Define the (unweighted) combinatorial discrepancy vector \(\bfd^{t,\D}\in\mathbb{Z}^{|\mathcal{I}|}\) by
\[
(\bfd^{t,\D})_B := \disc_t(B), \qquad B\in\mathcal{I}.
\]
Define the weighted combinatorial discrepancy vector \(\bfwd^{t,\D}\in\mathbb{R}^{|\mathcal{I}|}\) by
\[
(\bfwd^{t,\D})_B := \wdisc_t(B):=\gamma(B)\,\disc_t(B), \qquad B\in\mathcal{I},
\]
where $\gamma(B)=\prod_{j\in [d]: B_j\ne (0,1]} \gamma_j$ is the product weight defined in \eqref{eq:box_weight}.

\medskip
The continuous discrepancy and combinatorial discrepancy are tightly connected through the transference principle, which enables the construction of low continuous discrepancy sets by recursively applying colorings that minimize combinatorial discrepancy. This connection is made precise in~\cite[Lemma 3.4]{BJ25}:

\begin{lemma}[{Transference Principle, \cite[Lemma 3.4]{BJ25}}]\label{lem:transference}
For any measurable set $C \subseteq [0,1]^d$, we have
\[
h_T(C) = h_0(C) + \sum_{t=0}^{T-1} \frac{\disc_t(C)}{n_t}.
\]
\end{lemma}

By Lemma \ref{lem:transference}, we can control the integration error using the final point set \(A_T\) by bounding the error at each intermediate step \(t\in\{0,\ldots,T-1\}\).

Moreover, the \(\wSubgTrans\) algorithm ensures subgaussianity of both the coloring vector \(\bc_t\) and the weighted combinatorial discrepancy vector \(\bfwd^{t,\D}\) as follows:
\begin{lemma}[Subgaussianity for Dyadic Boxes] \label{lem:subg_dyadic}
For any $0 \leq t \leq T-1$, $\wSubgTrans$ satisfies (1) $\disc_t([0,1]^d) = 0$, and (2) the random vector $(\bc_t, \bfwd^{t,\D})$ is $O\big(d \log (dn)\prod_{j\in [d]}(1+\gamma_j^2\log n)\big)$-subgaussian\footnote{A random vector $\by \in \R^m$ is called $\sigma^2$-subgaussian if for all $\bz \in \R^m$, one has $\E[\exp(\langle \by, \bz\rangle)] \leq \exp(\sigma^2 \|\bz\|_2^2/2)$. }. 
\end{lemma}
\begin{proof}
The first claim follows because $\SubgVecBal$ always produces a balanced coloring.

The subgaussian property is based on the following result for the $\SubgVecBal$ algorithm:
\begin{theorem}[\cite{ALS21}] \label{thm:self_bal_walk}
Given vectors $\bv_1, \ldots, \bv_n \in \R^m$ with $\|\bv_j\|_2 \leq 1$ online, the \textsf{Self-Balancing Walk} computes a symmetric random coloring $\bc=(c_1,\ldots,c_n) \in \{\pm 1\}^n$, 
such that $\sum_{j =1}^t c_j \bv_j$ is $O(\log (mn))$-subgaussian for all $t \in [n]$ whp. 
Moreover, the algorithm runs in $O(\sum_{j \in [n]} \nnz(\bv_j))$ time.\footnote{$\nnz(\bz)$ denotes  number of non-zero coordinates in a vector $\bz$.} 
\end{theorem}
For every point $\bz_j \in A_t$ and every subset $u\subseteq [d]$, the point $\bz_j$ lies in exactly $h^{|u|}$ dyadic boxes in $\D_{\leq h}^{\otimes d}$ whose set of the nontrivial coordinate directions is precisely $u$. Thus 
 \begin{align*}
\|\bv_j\|_2^2 &=1+\sum_{B\in  \D_{\leq h}^{\otimes d} } w^2(B) \ind_{\{\bz_j \in B\}}\\
&=1+\sum_{B\in  \D_{\leq h}^{\otimes d}} \ind_{\{\bz_j \in B\}} \prod_{j\in [d]: B_j\ne (0,1]} \gamma_j^2 \\
&=1+\sum_{u\subseteq [d]} h^{|u|} \prod_{j\in u}\gamma_j^2 \\
&=1+\prod_{j\in [d]}(1+h\gamma_j^2),
\end{align*}
and the claim follows from Theorem \ref{thm:self_bal_walk} on the at most $n^2$ vectors $\bv_j$, $m=O(2^{(h+1)d})$ and $h= O(\log (dn))$. 
\end{proof}

\subsubsection{Proof of Theorem \ref{thm:main1}}

As noted in \cite[Remark 1.1]{BJ25}, requiring that $f$ has a Fourier series uses some (mild) smoothness assumptions, and it also requires $1$-periodicity, but this is without loss of generality via standard folding. These can be removed by exploiting that our method uses random samples and that the Fourier series of $f$ converges to $f$ in a $\ell_2$ sense. We ignore these technicalities as they are standard and do not give any new insight.

We are now ready to prove Theorem \ref{thm:main1}.

\begin{proof}[Proof of Theorem \ref{thm:main1}]

Fix an arbitrary 1-periodic function $f \in L^2([0,1]^d)$ with a continuous mixed derivative.
We can equivalently view $\wSubgTrans$ as fixing $\D_{\leq h}^{\otimes d}$ and applying a random shift $\bs \sim [0,1)^d$ to $f$ which results in $f_{\bs}(\bz) = f(\bz + \bs)$. Denote $\E_{\bs}$ the expectation w.r.t. random shift $\bs$. By Hlawka-Zaremba formula, the integration error using point set $A_T$ is
\begin{align*}
\err(A_T,f_{\bs}) = \sum_{\emptyset \neq u \subseteq [d]} (-1)^{|u|-1} \int_{Q_u} h_T(\bz) \cdot \partial_u f_{\bs}(\bz)  d \bz ,
\end{align*}
where $h_T(\bz) = \vol(C_{\bz}) - |C_{\bz}\cap A_T|/n_T$ is the (scaled) continuous discrepancy of the corner $C_{\bz}$,  $Q_u:=\{\bz\in [0,1]^d:z_j=1 \text{ for all } j\notin u\}$ denotes the (shifted) $|u|$-dimensional cube in $\R^d$, and $\partial_u f_{\bs}(\bz) = \frac{\partial^{|u|} f_{\bs}(\bz)}{\partial \bz_u}$. By Lemma \ref{lem:transference}, we can write
\[
h_T(\bz) = h_0(\bz) + \sum_{t=0}^{T-1} \frac{\disc_t(\bz)}{n_t}.
\]
Substituting this expression into the error representation yields
\begin{align}\label{eq:error_decomp_high_dim}
    \err(A_T,f_{\bs}) &= \sum_{\emptyset \neq u \subseteq [d]} (-1)^{|u|-1} \int_{Q_u} h_T(\bz) \cdot \partial_u f_{\bs}(\bz)  d \bz \nonumber\\
    &=\sum_{\emptyset \neq u \subseteq [d]} (-1)^{|u|-1} \int_{Q_u} h_0(\bz) \cdot \partial_u f_{\bs}(\bz)  d \bz + \sum_{t=0}^{T-1}\Bigg(\sum_{\emptyset \neq u \subseteq [d]} (-1)^{|u|-1} \int_{Q_u} \frac{\disc_t(\bz)}{n_t} \cdot \partial_u f_{\bs}(\bz)  d \bz\Bigg) \nonumber\\
    &=:\err(A_0,f_{\bs}) + \sum_{t=0}^{T-1} \err^{\disc}_t(f_{\bs}).
\end{align}

First, since $A_0$ consists of $n^2$ independent random samples,
\begin{align}\label{eq:proof_aux1}
\E[(\err(A_0,f_{\bfs}))^2]\le \frac{\sigma(f)^2}{n^2}.
\end{align}
We next analyze the discrepancy error $\err^{\disc}_t$ in \eqref{eq:error_decomp_high_dim}. Observe that
\begin{align}\label{eq:disct-d}
    \err^{\disc}_{t}(f_{\bs}) & =\sum_{\emptyset \neq u \subseteq [d]} (-1)^{|u|-1} \int_{Q_u} \frac{\disc_{t}(\bz)}{n_t} \cdot \partial_u f_{\bs}(\bz) d \bz \nonumber\\
    &= \sum_{\emptyset \neq u \subseteq [d]} (-1)^{|u|-1} \int_{Q_u} \frac{\gamma(u)\disc_{t}(\bz)}{n_t} \cdot \frac{\partial_u f_{\bs}(\bz)}{\gamma(u)} d \bz \nonumber\\
    &= \sum_{\emptyset \neq u \subseteq [d]} (-1)^{|u|-1} \int_{Q_u} \frac{\wdisc_{t}(C_{\bz})}{n_t} \cdot \frac{\partial_u f_{\bs}(\bz)}{\gamma(u)} d \bz,
\end{align}
where $\gamma(u)=\prod_{j\in u}\gamma_j$ is the product weight.

At a high level, we analyze the contribution from each $\emptyset \neq u \subseteq [d]$ using tensorization, and then exploit the joint subgaussianity of discrepancy on each $Q_u$ to combine them. 

Let us fix $\emptyset \neq u \subseteq [d]$ and consider the integral in \eqref{eq:disct-d} for $u$. Let us assume wlog that the combinatorial discrepancy $\disc_t(\bz)$ on $Q_u$ is a step function with values changing only when any coordinate reaches $2^{-h} \Z \cap [0,1)$. 
This way, we may view $\disc_t(\bz)$ restricted to $Q_u$ as a vector\footnote{We have ignored the $0$-measure set of corners $C \subseteq Q_u$ with some dimension being $(0,1]$.} $\bfd^{t, \mathcal{C}, u} \in \Z^{|\mathcal{C}_h^{\otimes u}|}$ with coordinates\footnote{Strictly speaking, the index $\bfj$ here should really range over $\{0, \ldots, 2^h-1\}^u \times \{2^h\}^{\overline{u}}$ so that the coordinates in $\overline{u}$ of $\bfj/2^h$ are $1$. But to keep notation clear, even though we are working with $d$-dimensional vectors we only display the coordinates of $\bfj$ in $u$ 
which are the only ones that can vary. The same convention applies when we talk about corners $\mathcal{C}_h^{\otimes u}$, dyadic boxes $\mathcal{D}_{\leq h}^{\otimes u}$, and decomposition $P_h^{\otimes u}$.} 
$\bfd^{t, \mathcal{C}, u}_{\bfj} = \disc_t(\bfj/2^h)$ for all $\bfj \in \{0, \ldots, 2^h-1\}^u$ corresponding to corners $\mathcal{C}_h^{\otimes u}$.  Similarly, we may view $\wdisc_t(\bz)$ restricted to $Q_u$ as a vector $\bfwd^{t, \mathcal{C}, u} \in \Z^{|\mathcal{C}_h^{\otimes u}|}$ with coordinates $\bfwd^{t, \mathcal{C}, u}_{\bfj} = \wdisc_t(\bfj/2^h)$ for all $\bfj \in \{0, \ldots, 2^h-1\}^u$ corresponding to corners $\mathcal{C}_h^{\otimes u}$.

We also define the vector $\by^{f_{\bs},u} \in \R^{|\mathcal{C}_h^{\otimes u}|}$ as
\[
\by^{f_{\bs},u}_{\bfj} := \int_{[\bfj/2^h, (\bfj+ \ind_u)/2^h] } \partial_u f_{\bs}(\bz^u, \ind_{\overline{u}}) \cdot d \bz^u. 
\]
for all $\bfj \in \{0, \ldots, 2^h-1\}^u$. It is possible to further express $\by^{f_{\bs},u}_{\bfj}$ as a sum of function values at the vertices of $[\bfj/2^h, (\bfj+ \ind_u)/2^h] \times \ind_{\overline{u}}$ with mixed signs, but we do not need such a formula. 
Next, observe that the vector $\by^{f_{\bs},u}$ takes a tensor form if the function $f_{\bs}$ is the product of functions of each coordinate. 

\begin{observation}[Tensor Form for Product Functions]
\label{obs:tensor_form_u}
If $f_{\bs}: [0,1]^d \rightarrow \C$ has the form $f_{\bs}(\bz) = \prod_{i \in u} g_i(z_i)$ for all $\bz \in Q_u$, then we have $\by^{f_{\bs},u}_{\bfj} = \prod_{i \in u} \by^{g_i}_{\bfj_i}$ for any $\bfj \in \{0, \ldots, 2^h-1\}^u$, where $\by^{g_i}$ is defined as in \cite{BJ25}. Consequently, we have $\by^{f_{\bs},u} = \otimes_{i \in u} \by^{g_i}$.
\end{observation}

This observation is immediate from the definition of $\by^{f_{\bs},u}$ above. Using dyadic decomposition, we may now express $\err_t^{\disc}$ as inner products
\begin{align} \label{eq:disc_err_high_dim}
\err^{\disc}_t(f_{\bs})
&=\sum_{\emptyset \neq u \subseteq [d]} (-1)^{|u|-1} \int_{Q_u} \frac{\wdisc_t(C_{\bz})}{n_t} \cdot \frac{\partial_u f_{\bs}(\bz)}{\gamma(u)} d \bz \nonumber\\
&= \frac{1}{n_t} \sum_{\emptyset \neq u \subseteq [d]}  (-1)^{|u|-1} \cdot \big\langle \bfwd^{t, \mathcal{C}, u} , \frac{\by^{f_{\bs},u}}{\gamma(u)} \big\rangle \nonumber \\
& = \frac{1}{n_t} \sum_{\emptyset \neq u \subseteq [d]}  (-1)^{|u|-1} \cdot \big\langle \bfwd^{t, \D, u} , (P_h^\top)^{\otimes u} \frac{\by^{f_{\bs},u}}{\gamma(u)} \big\rangle,
\end{align}
where $\bfwd^{t, \D, u} \in \Z^{|(\D_{\leq h} \setminus C_1)^{\otimes u}|}$ is the restriction of $\bfwd^{t,\D}$ to dyadic boxes in $Q_u$ excluding those with any dimension in $u$ being $(0,1]$.  It follows that the vectors $\bfwd^{t, \D, u}$ for different $\emptyset \neq u \subseteq [d]$ correspond to disjoint coordinates of $\bfwd^{t,\D}$. Consequently, by \Cref{lem:subg_dyadic},  for any outcome of the random shift $\bs$, $\err^{\disc}_t(f_{\bs})$ is also subgaussian with parameter 
\begin{align}\label{eq:disc_err_subg_him_dim}
&O\bigg(\frac{d\log (dn)\prod_{j\in [d]}(1+\gamma_j^2\log n)}{n_t^2} \cdot \sum_{\emptyset \neq u \subseteq [d]} \frac{1}{\gamma(u)^2} \big\|(P_h^\top)^{\otimes u} \by^{f_{\bs},u} \big\|_2^2\bigg) \nonumber \\
&\leq O\bigg(\frac{d\log (dn)\prod_{j\in [d]}(1+\gamma_j^2\log n)}{n_t^2} \cdot \sum_{\emptyset \neq u \subseteq [d]}  \frac{1}{\gamma(u)^2}  \big\|(\overline{P}_h^\top)^{\otimes u} \by^{f_{\bs},u} \big\|_2^2\bigg).
\end{align}

Therefore, we are left to derive an upper bound on  $\sum_{\emptyset \neq u \subseteq [d]} \frac{1}{\gamma(u)^2}\big\|(\overline{P}_h^\top)^{\otimes u} \by^{f_{\bs},u} \big\|_2^2$, for which we look at the Fourier transformation of $f_{\bs}$.

\smallskip
\noindent \textbf{Bounding the $\ell_2$ Norm via Fourier.} 
We may write 
\[
f_{\bs}(\bz) = \sum_{\bk \in \Z^d}\widehat{f_{\bs}}(\bk) \cdot e^{2 \pi i \langle \bk , \bz \rangle} =: \sum_{\bk \in \Z^d}\widehat{f_{\bs}}(\bk) \cdot e_{\bk}(\bz) ,
\]
where $e_{\bk}(\bz) := \exp(2 \pi i \langle \bk, \bz\rangle)$. 
Again, by replacing $f_{\bs}$ with $f_{\bs} - \widehat{f}(\mathbf{0})$, we may assume wlog that $\widehat{f_{\bs}}(\mathbf{0}) = 0$. Then we may write $\by^{f_{\bs},u} = \sum_{\bk \in \Z^d \setminus \mathbf{0}} \widehat{f_{\bs}}(\bk) \cdot \by^{e_{\bk}, u}$. 
Note that the vector $\by^{e_{\bk}, u}$ are the same for the set of $e_{\bk}$ for which $k_u$ coincide.

The following orthogonality property of the high dimensional Fourier coefficients is a simple corollary of its one dimensional version in \cite{BJ25}.

\begin{corollary}[Fourier Orthogonality]
\label{cor:orthogonality_fourier_high_dim}
For any $\bk \neq \bk' \in \mathbb{Z}^d$, we have $\E_{\bs}\big[ \widehat{f_{\bs}}(\bk)^* \widehat{f_{\bs}}(\bk')\big] = 0$. 
\end{corollary}

Now we can compute the $\ell_2$ norm as follows.
\begin{align*}
\E_{\bs} \Big[\Big\| (\overline{P}^\top_h)^{\otimes u} \cdot \by^{f_{\bs},u} \Big\|_2^2\Big] 
& = \E_{\bs} \Big[\Big| (\overline{P}^\top_h)^{\otimes u} \cdot \Big( \sum_{\bk\neq \mathbf{0}} \widehat{f_{\bs}}(\bk) \by^{e_{\bk},u} \Big) \Big|^2 \Big] \\
& = \sum_{\bk \neq \mathbf{0}} \big | \widehat{f}(\bk) \big|^2 \cdot \prod_{j \in u} |\overline{P}^\top_h \by^{e_{k_j}} |^2 
\leq  \sum_{\bk\neq \mathbf{0}} \big|\widehat{f}(\bk) \big|^2  \cdot \prod_{j \in u} |k_j| ,
\end{align*}
where the expectation over $\bs$ is removed in the second line since the modulus of the Fourier coefficients do not depend on the random shift, and the last step follows from the estimate in \cite[Lemma 5.3]{BJ25}.

Now summing over all $\emptyset \neq u \subseteq [d]$ gives the following $\ell_2$ bound
\begin{align} \label{eq:ell_2_bound_high_dim}
\E_{\bs} \Big[\sum_{\emptyset \neq u \subseteq [d]} \frac{1}{\gamma(u)^2}\Big\| (\overline{P}^\top_h)^{\otimes u}  \by^{f_{\bs},u} \Big\|_2^2 \Big]
& \leq \sum_{\bk\neq \mathbf{0}} \big|\widehat{f}(\bk) \big|^2 \bigg(\sum_{\emptyset \neq u \subseteq [d]} \frac{1}{\gamma(u)^2}   \prod_{j \in u} |k_j| \bigg) \nonumber\\
&=\sum_{\bk\neq \mathbf{0}} \big|\widehat{f}(\bk) \big|^2 \bigg(\sum_{\emptyset \neq u \subseteq [d]}   \prod_{j \in u} \frac{|k_j|}{\gamma_j^2} \bigg) \nonumber \\
&=:\sigma_{\WSO}(f)^2 .
\end{align}
Using the $\ell_2$-norm bound \eqref{eq:ell_2_bound_high_dim} in \eqref{eq:disc_err_high_dim} and \eqref{eq:disc_err_subg_him_dim}, we obtain the bound
\begin{align} \label{eq:disc_var_high_dim}
\E \Big[ \big(\sum_{t=0}^{T-1} \err^{\disc}_t(f_{\bs})\big)^2 \Big] 
\leq O\bigg(\frac{d\log (dn)\prod_{j\in [d]}(1+\gamma_j^2\log n)}{n^2} \cdot \sigma_{\WSO}(f)^2 \bigg).
\end{align}

Plugging \eqref{eq:disc_var_high_dim} and \eqref{eq:proof_aux1} into \eqref{eq:error_decomp_high_dim}, we obtain that
\begin{align*}
\E[(\err(A_T,f_{\bfs}))^2] & \leq \E[(\err(A_0,f_{\bfs}))^2] + \E \Big[ \big(\sum_{t=0}^{T-1} \err^{\disc}_t(f_{\bs})\big)^2 \Big] \\
& \leq  \frac{\sigma(f)^2}{n^2}+ O\bigg(\frac{d\log (dn)\prod_{j\in [d]}(1+\gamma_j^2\log n)}{n^2} \cdot \sigma_{\WSO}(f)^2 \bigg) \\
&= O\bigg(\frac{d\log (dn)\prod_{j\in [d]}(1+\gamma_j^2\log n)}{n^2} \cdot \sigma_{\WSO}(f)^2 \bigg).
\end{align*}
This is the error guaranteed by the algorithm in \cite{BJ25}, applied for the weighted discrepancy. This completes the proof of the theorem. 
\end{proof}

\subsection{Proofs of results in Section \ref{sec:effective_dimension}}\label{app:proof_2}

\begin{proof}[Proof of Theorem \ref{thm:superposition}]

Let $\F_{s}$ be the class of functions involving interactions of order at most $s$, i.e., $\F_{s}$ consists functions of the form $\sum_{|u|\le s} f_{u}(\bz^u)$. Fix an arbitrary $1$-periodic function \[
f(\bz)=\sum_{|u|\le s} f_{u}(\bz^u)\in \F_{s}
\]
with a continuous mixed derivative. We can equivalently view $\wSubgTrans$ as fixing $\D_{\leq h}^{\otimes d}$ and applying a random shift $\bs \sim [0,1)^d$ to $f$ which results in $f_{\bs}(\bz) = f(\bz + \bs)=\sum_{|u|\le s} f_{u}(\bz^u+\bs^{u})$. Denote $\E_{\bs}$ the expectation w.r.t. random shift $\bs$. By Hlawka-Zaremba formula, the integration error using point set $A_T$ is
\begin{align*}
\err(A_T,f_{\bs}) &= \sum_{\emptyset \neq u \subseteq [d]} (-1)^{|u|-1} \int_{Q_u} h_T(\bz) \cdot \partial_u f_{\bs}(\bz)  d \bz \\
&=\sum_{\emptyset \neq u \subseteq [d]} (-1)^{|u|-1} \int_{Q_u} h_T(\bz) \cdot \partial_u \bigg(\sum_{|u'|\le s} f_{u'}(\bz^{u'}+\bs^{u'})\bigg)  d \bz\\
&=\sum_{0<|u|\le s} (-1)^{|u|-1} \int_{Q_u} h_T(\bz) \cdot \partial_u \bigg(\sum_{u': u\subseteq u'} f_{u'}(\bz^{u'}+\bs^{u'})\bigg)  d \bz.
\end{align*}
By Lemma \ref{lem:transference}, we can write
\begin{align}\label{eq:error_decomp_high_dim_2}
\err(A_T,f_{\bs})
&=\sum_{0<|u|\le s} (-1)^{|u|-1} \int_{Q_u} h_T(\bz) \cdot \partial_u \bigg(\sum_{u': u\subseteq u'} f_{u'}(\bz^{u'}+\bs^{u'})\bigg)  d \bz \nonumber\\
&=\sum_{0<|u|\le s} (-1)^{|u|-1} \int_{Q_u} h_0(\bz) \cdot \partial_u \bigg(\sum_{u': u\subseteq u'} f_{u'}(\bz^{u'}+\bs^{u'})\bigg)  d \bz \nonumber\\
&\quad +\sum_{t=0}^{T-1}\bigg(\sum_{0<|u|\le s} (-1)^{|u|-1} \int_{Q_u} \frac{\disc_t(\bz)}{n_t} \cdot \partial_u \bigg(\sum_{u': u\subseteq u'} f_{u'}(\bz^{u'}+\bs^{u'})\bigg)  d \bz\bigg) \nonumber\\
&=:\err(A_0,f_{\bs}) + \sum_{t=0}^{T-1} \err^{\disc}_t(f_{\bs}).
\end{align}

First, since $A_0$ consists of $n^2$ independent random samples,
\begin{align}\label{eq:proof2_aux1}
\E[(\err(A_0,f_{\bfs}))^2]\le \frac{\sigma(f)^2}{n^2}.
\end{align}

We next analyze the discrepancy error $\err^{\disc}_t$ in \eqref{eq:error_decomp_high_dim_2}. Using the same notation and argument as in the proof of Theorem \ref{thm:main1} in Appendix \ref{app:proof_1}, we further write
\begin{align}\label{eq:disct-d_2}
    \err^{\disc}_{t}(f_{\bs}) &=\sum_{0<|u| \le s}\sum_{u': u \subseteq u'} (-1)^{|u|-1} \int_{Q_u} \frac{\disc_t(\bz)}{n_t} \cdot \partial_u  f_{u'}(\bz^{u'}+\bs^{u'})  d \bz\nonumber\\
    &=\frac{1}{n_t} \sum_{0<|u| \le s}\sum_{u': u \subseteq u'}  (-1)^{|u|-1} \cdot \big\langle \bfd^{t, \mathcal{C}, u} ,\by^{f_{u',\bs},u} \big\rangle \nonumber \\
& = \frac{1}{n_t} \sum_{0<|u| \le s}\sum_{u': u \subseteq u'}  (-1)^{|u|-1} \cdot \big\langle \bfd^{t, \D, u} , (P_h^\top)^{\otimes u} \by^{f_{u',\bs},u} \big\rangle\nonumber\\
&=  \frac{1}{n_t} \sum_{0<|u|\le s} (-1)^{|u|-1}    \cdot \big\langle \bfd^{t, \D, u} , (P_h^\top)^{\otimes u} (\sum_{u':u \subseteq u'}\by^{f_{u',\bs},u} )\big\rangle,
\end{align}
where $f_{u',\bs}(\bz):=f_{u'}(\bz^{u'}+\bs^{u'})$, and $\bfd^{t, \D, u} \in \Z^{|(\D_{\leq h} \setminus C_1)^{\otimes u}|}$ is the restriction of $\bfd^{t,\D}$ to dyadic boxes in $Q_u$ excluding those with any dimension in $u$ being $(0,1]$. It follows that the vectors $\bfd^{t, \D, u}$ for different $u$ with $0<|u|\le s$ correspond to disjoint coordinates of the vector $\bfd^{t,\D}|_{\widetilde{\D}_{\leq h}^{\otimes s}}$, which is the restriction of $\bfd^{t,\D}$  to $\widetilde{\D}_{\leq h}^{\otimes s}$.

Recall that in our modified $\wSubgTrans$ algorithm, the combinatorial discrepancy vector is defined with respect to $\widetilde{\D}_{\leq h}^{\otimes s}$ rather than the full collection $\D_{\leq h}^{\otimes d}$. Moreover,
\begin{align*}
\|\bv_j\|_2^2 =1+\sum_{B\in  \widetilde{\D}_{\leq h}^{\otimes s} } \ind_{\{\bz_j \in B\}}=1+\sum_{u\subseteq [d]:|u|\le s} h^{|u|} =1+\sum_{k=0}^{s}\binom{d}{k}h^{k}=O_s(d^sh^s)=O_{s}\big((d\log (dn))^s\big).
\end{align*}
Analogous to Lemma~\ref{lem:subg_dyadic}, we conclude that for any $0\le t\le T-1$, the random vector $(\bc_t, \bfd^{t,\D}|_{\widetilde{\D}_{\leq h}^{\otimes s}})$ is $O_s\big(d^s (\log (dn))^{s+1} \big)$-subgaussian.

Consequently, for any outcome of the random shift $\bs$, $\err^{\disc}_t(f_{\bs})$ in \eqref{eq:disct-d_2} is also subgaussian with parameter 
\begin{align}\label{eq:disc_err_subg_him_dim_2}
&O_s\bigg(\frac{d^s (\log (dn))^{s+1}}{n_t^2} \cdot \sum_{0<|u|\le s}  \big\|(P_h^\top)^{\otimes u} (\sum_{u':u \subseteq u'}\by^{f_{u',\bs},u} ) \big\|_2^2\bigg) \nonumber \\
&\leq O_s\bigg(\frac{d^s (\log (dn))^{s+1}}{n_t^2} \cdot \sum_{0<|u|\le s}   \big\|(\overline{P}_h^\top)^{\otimes u} (\sum_{u':u \subseteq u'}\by^{f_{u',\bs},u} ) \big\|_2^2\bigg).
\end{align}

For a fixed subset $u\subseteq [d]$ with $0<|u|\le s$, we have 
\begin{align*}
\E_{\bs} \Big[ \big\|(\overline{P}_h^\top)^{\otimes u} (\sum_{u':u \subseteq u'}\by^{f_{u',\bs},u} ) \big\|_2^2\Big] 
& = \E_{\bs} \Big[\Big| (\overline{P}^\top_h)^{\otimes u} \cdot \Big( \sum_{u':u \subseteq u'}\sum_{\bk\neq \mathbf{0}} \widehat{f_{u',\bs}}(\bk) \by^{e_{\bk},u} \Big) \Big|^2 \Big] \\
&= \E_{\bs} \Big[\Big| (\overline{P}^\top_h)^{\otimes u} \cdot \Big( \sum_{\bk\neq \mathbf{0}} \Big(\sum_{u':u \subseteq u'}\widehat{f_{u',\bs}}(\bk)\Big) \by^{e_{\bk},u} \Big) \Big|^2 \Big]\\
& = \sum_{\bk \neq \mathbf{0}} \big | \sum_{u':u \subseteq u'}\widehat{f_{u'}}(\bk) \big|^2 \cdot \prod_{j \in u} |\overline{P}^\top_h \by^{e_{k_j}} |^2 \\
&\leq  \sum_{\bk\neq \mathbf{0}} \big|\sum_{u':u \subseteq u'}\widehat{f_{u'}}(\bk) \big|^2  \cdot \prod_{j \in u} |k_j| ,
\end{align*}
where we have used the Fourier orthogonality (Corollary \ref{cor:orthogonality_fourier_high_dim}) and the estimate \cite[Lemma 5.3]{BJ25}, and the expectation over $\bs$ is removed in the third line since the modulus of the Fourier coefficients do not depend on the random shift.

Now summing over all subsets $u\subseteq [d]$ satisfying $0<|u|\le s$ gives the following $\ell_2$ bound
\begin{align}\label{eq:ell_2_bound_high_dim_2}
\E_{\bs} \Big[ \sum_{0<|u|\le s}\big\|(\overline{P}_h^\top)^{\otimes u} (\sum_{u':u \subseteq u'}\by^{f_{u',\bs},u} ) \big\|_2^2\Big]\le  \sum_{0<|u|\le s} \sum_{\bk\neq \mathbf{0}} \big|\sum_{u':u \subseteq u'}\widehat{f_{u'}}(\bk) \big|^2  \cdot \prod_{j \in u} |k_j|.
\end{align}

For $\bk=(k_1,\ldots,k_d)\in \Z^d$, we denote its support by $\mathrm{supp}(\bk):=\{j\in [d]:k_j\neq 0\}$. By the definition of Fourier transform, $\widehat{f_{u'}}(\bk)=0$ if $\mathrm{supp}(\bk)$ is not a subset of $u'$. Therefore,
\begin{align}\label{eq:proof2_aux2}
    \sum_{0<|u|\le s} \sum_{\bk\neq \mathbf{0}} \big|\sum_{u':u \subseteq u'}\widehat{f_{u'}}(\bk) \big|^2  \cdot \prod_{j \in u} |k_j|&=  \sum_{\bk\neq \mathbf{0}} \sum_{0<|u|\le s} \big|\sum_{u':u \subseteq u'}\widehat{f_{u'}}(\bk) \big|^2  \cdot \prod_{j \in u} |k_j| \nonumber\\
    &=\sum_{\bk\neq \mathbf{0}} \sum_{0<|u|\le s, u\subseteq \mathrm{supp}(\bk) } \big|\sum_{u':u \subseteq u', \mathrm{supp}(\bk)\subseteq u' }\widehat{f_{u'}}(\bk) \big|^2  \cdot \prod_{j \in u} |k_j| \nonumber\\
    &=\sum_{\bk\neq \mathbf{0}} \sum_{0<|u|\le s, u\subseteq \mathrm{supp}(\bk) } \big|\sum_{u':\mathrm{supp}(\bk)\subseteq u' }\widehat{f_{u'}}(\bk) \big|^2  \cdot \prod_{j \in u} |k_j| \nonumber\\
    &=\sum_{\bk\neq \mathbf{0}} \big|\sum_{u':\mathrm{supp}(\bk)\subseteq u' }\widehat{f_{u'}}(\bk) \big|^2  \bigg(\sum_{0<|u|\le s, u\subseteq \mathrm{supp}(\bk) }  \prod_{j \in u} |k_j|\bigg) \nonumber\\
    &\overset{(\star)}{=}\sum_{\bk\neq \mathbf{0}} \big|\sum_{u':\mathrm{supp}(\bk)\subseteq u' }\widehat{f_{u'}}(\bk) \big|^2\bigg(\sum_{\emptyset \neq u \subseteq [d]}   \prod_{j \in u}|k_j|\bigg) \nonumber\\
    &=\sum_{\bk\neq \mathbf{0}} \big|\sum_{|u'|\le s}\widehat{f_{u'}}(\bk) \big|^2\bigg(\sum_{\emptyset \neq u \subseteq [d]}   \prod_{j \in u}|k_j|\bigg) \nonumber\\
    &=\sum_{\bk\neq \mathbf{0}} \big|\widehat{f}(\bk) \big|^2\bigg(\sum_{\emptyset \neq u \subseteq [d]}   \prod_{j \in u}|k_j|\bigg) \nonumber\\
    &=: \sigma_{\SO}(f)^2,
\end{align}
where $(\star)$ follows because the sum over $\bk\neq \mathbf{0}$ effectively restricts to frequencies with $|\mathrm{supp}(\bk)|\le s$. Hence, $\sum_{0<|u|\le s, u\subseteq \mathrm{supp}(\bk) }  \prod_{j \in u} |k_j|=\sum_{\emptyset \neq u\subseteq \mathrm{supp}(\bk) }  \prod_{j \in u} |k_j|=\sum_{\emptyset \neq u \subseteq [d]}   \prod_{j \in u}|k_j|$.

Using the $\ell_2$-norm bound \eqref{eq:ell_2_bound_high_dim_2} and the equality \eqref{eq:proof2_aux2} in \eqref{eq:disct-d_2} and \eqref{eq:disc_err_subg_him_dim_2}, we obtain the bound
\begin{align} \label{eq:disc_var_high_dim_2}
\E \Big[ \big(\sum_{t=0}^{T-1} \err^{\disc}_t(f_{\bs})\big)^2 \Big] 
\leq O_s\bigg(\frac{d^s (\log (dn))^{s+1}}{n^2} \cdot \sigma_{\SO}(f)^2 \bigg).
\end{align}

Plugging \eqref{eq:disc_var_high_dim_2} and \eqref{eq:proof2_aux1} into \eqref{eq:error_decomp_high_dim_2}, we obtain that for any $f(\bz)=\sum_{|u|\le s} f_{u}(\bz^u)\in \F_s$,
\begin{align}\label{eq:err_h}
\E[(\err(A_T,f_{\bfs}))^2] & \leq \E[(\err(A_0,f_{\bfs}))^2] + \E \Big[ \big(\sum_{t=0}^{T-1} \err^{\disc}_t(f_{\bs})\big)^2 \Big] \nonumber \\
& \leq  \frac{\sigma(f)^2}{n^2}+ O_s\bigg(\frac{d^s (\log (dn))^{s+1}}{n^2} \cdot \sigma_{\SO}(f)^2 \bigg) \nonumber\\
&= O_s\bigg(\frac{d^s (\log (dn))^{s+1}}{n^2} \cdot \sigma_{\SO}(f)^2 \bigg).
\end{align}

A slight modification of \cite[Lemma 4.3]{BJ25} yields that, for any $f:[0,1]^d\to \R$,
\begin{align}\label{eq:err_g}
\E[(\err(A_T,f))^2]\le O\Big(\frac{\log n}{n}\sigma(f)^2\Big).
\end{align}

For any decomposition $f=g+h$ with $g,h: [0,1]^d\to \R$ and $h\in \F_{s}$, we have $\err(A_T,f)=\err(A_T,g)+\err(A_T,h)$. The theorem then follows directly from \eqref{eq:err_h} and \eqref{eq:err_g}.
\end{proof}

\begin{proof}[Proof of Theorem \ref{thm:truncation}]
   In this case, $\gamma_1=\cdots =\gamma_s=1,\gamma_{s+1}=\cdots=\gamma_{d}=0$, and thus
\[
\|\bv_j\|_2^2=1+\prod_{j\in [d]}(1+h\gamma_j^2)=O(h^{s})=O((\log (dn))^s).
\]
Analogous to Lemma~\ref{lem:subg_dyadic}, we conclude that for any $0\le t\le T-1$, the random vector $(\bc_t, \bfwd^{t,\D})$ is $O_s((\log (d n))^{s+1})$-subgaussian. The result follows from a slight modification of the arguments in the proofs of Theorems \ref{thm:main1} and \ref{thm:superposition}; we omit the proof for brevity.
\end{proof}

\end{document}